\documentclass{amsart}

\usepackage[latin1]{inputenc}
\usepackage[T1]{fontenc}
\usepackage{enumerate}
\usepackage{diagbox}  

\usepackage[english]{babel}

\usepackage{amsmath,amsfonts,amssymb,amsthm,amscd}

\newtheorem{theorem}{Theorem}[section]
\newtheorem{lemma}[theorem]{Lemma}
\newtheorem{corollary}[theorem]{Corollary}
\newtheorem{proposition}[theorem]{Proposition}
\newtheorem{claim}{Claim}

\theoremstyle{definition}

\newtheorem{example}[theorem]{Example}

\newtheorem{question}{Question}

\theoremstyle{remark}
\newtheorem{remark}[theorem]{Remark}

\numberwithin{equation}{section}

\newcommand{\N}{\mathbb N}

\newcommand{\R}{\mathbb R}
\newcommand{\Z}{\mathbb Z}
\newcommand{\C}{\mathbb C}
\newcommand{\D}{\mathbb D}
\newcommand{\bigo}{\mathcal{O}}

\title[Hypercyclic algebras for convolution operators]{Hypercyclic algebras for convolution operators of unimodular constant term}
\author[J.\ B\`{e}s, R.\ Ernst, and A.\ Prieto]{J.\ B\`{e}s, R.\ Ernst, and A.\ Prieto}
\address{J.~B\`{e}s, Department of Mathematics and Statistics,
Bowling Green State University,
Bowling Green, Ohio 43403,
USA.}
\email{jbes@bgsu.edu}
\address{R.~Ernst, LMPA, Centre Universitaire de la Mi-Voix, Maison de la Recherche Blaise-Pascal, 50 Rue Ferdinand Buisson, BP 699, 62228 Calais Cedex, France.
 }
\email{romuald.ernst@math.cnrs.fr}
\address{A.~Prieto, Departamento de An\'{a}lisis Matem\'{a}tico y Matem\'atica Aplicada,
Universidad Complutense de Madrid, Plaza de Ciencias 3, 28040
Madrid,
Spain}
\email{angelin@mat.ucm.es}
\thanks{This work is supported in part by MEC, Project
MTM 2016-7963-P, Programme PEPS JC 2017 INSMI, and by MEC Grant MTM 2015-65825-P}

\date{May 8th, 2019}
\subjclass[2010]{Primary 47A16, 46E10}
\keywords{Hypercyclic algebras; convolution operators, algebrability}

\allowdisplaybreaks[3]
\begin{document}
\begin{abstract}
We study  the existence of hypercyclic algebras for convolution operators $\Phi(D)$ on the space of entire functions whose symbol $\Phi$ has unimodular constant term. In particular, we provide new eigenvalue criteria for the existence of densely strongly algebrable sets of hypercyclic vectors.
\end{abstract}
\maketitle
\section{Introduction}
Once a linear dynamical system $(X, T)$ supports a hypercyclic vector, that is, a vector $x$ whose orbit $\{ T^nx\}_{n=0}^\infty$ is dense, the space $X$ always contains a dense linear subspace consisting entirely (but zero) of hypercyclic vectors~\cite{wengenroth2003hypercyclic}. When we further assume $X$ to be a topological algebra, it is natural to ask whether $T$ can support a {\em hypercyclic algebra,} that is, a subalgebra  of $X$ consisting entirely (but zero) of hypercyclic vectors. 

Particular attention has been given to this question for the case where $X=H(\mathbb{C})$ is the algebra of entire functions on the complex plane, endowed with the compact-open topology, and where $T$ is a convolution operator, that is, an operator that commutes with all translations. Godefroy and Shapiro~\cite{godefroy_shapiro1991operators} showed that such operators are precisely the ones that commute with the operator $D$ of complex differentiation, that they coincide with the operators of the form $T=\Phi(D)$ with $\Phi\in H(\mathbb{C})$ of exponential type, and that a convolution operator supports hypercyclic vectors precisely when it is not a scalar multiple of the identity.

Aron et al~\cite{aron_conejero_peris_seoane-sepulveda2007powers} first noted that any translation $\tau_a$, which is the operator $\Phi(D)$ with $\Phi(z)=e^{az}$, fails to support a hypercyclic algebra in a dramatic way: given any $f\in H(\mathbb{C})$, the multiplicity of any zero of an element in the orbit of $f^p$  must be divisible by $p$.  They also stopped short from giving a positive answer for $D$, showing that for the generic element $f$ of $H(\mathbb{C})$, each power $f^n$ $(n\in\N)$ is a hypercyclic vector for $D$. Shkarin~\cite{shkarin2010on}, and independently Bayart and Matheron~\cite{bayart_matheron2009dynamics} finally showed that $D$ supports a hypercyclic algebra, prompting the question:
\begin{question} \label{Q:1}
{\em Which convolution operators 
support a hypercyclic algebra? In other words, for which $\Phi\in H(\mathbb{C})$ of exponential type does $\Phi(D)$ support a hypercyclic algebra?}
\end{question}
Several new positive examples were obtained since then by various authors~\cite{bes_conejero_papathanasiou2016convolution,BCP2018,bernal-gonzalez_calderon-moreno}, including for instance when $\Phi$ is of subexponential growth having zero constant term, or of growth order one such as $\Phi(z)=\mbox{cos}(z)$ or $\Phi(z)=ze^z$, say.
Recently Bayart~\cite{bayart} provided a complete characterization for the case when $\Phi(0)$ has modulus strictly smaller than one, and provided a large class of examples when $\Phi$ is of subexponential growth and has unimodular constant term:

\begin{theorem} {\bf (Bayart)} \label{T:1.1} \cite[Theorem 1.1]{bayart}
Let $\Phi$ be a non-constant entire function of exponential type.
\begin{enumerate}
\item[{\rm 1.}] \ Assume that $|\Phi (0)|<1$. Then the following are equivalent:
\begin{enumerate}
\item[{\rm (i)}]\ $\Phi(D)$ supports a hypercyclic algebra.
\item[{\rm (ii)}]\ $\Phi$ is not a scalar multiple of an exponential function.
\end{enumerate}
\item[{\rm 2.}] \ Assume that $|\Phi (0)|=1$ and that $\Phi$ has subexponential growth. If either $\Phi'(0)\ne 0$ or $\Phi$ has order less than $1/2$, then $\Phi(D)$ supports a hypercyclic algebra.
\end{enumerate}
\end{theorem}




The purpose of this paper is to continue the study of hypercyclic algebras for convolution operators  $\Phi(D)$ with unimodular constant term $\Phi(0)$.  Notice that  it suffices to understand the case when $\Phi(0)=1$, as by the Le\'on-M\"{u}ller Theorem~\cite{LeoMul06} any operator $T$ has the same hypercyclic vectors as a scalar multiple $\lambda T$ whenever the scalar $\lambda$ has modulus one. We can say the following (see Subsection~\ref{Ss:Notation}):

\begin{theorem} \label{T:Ts}
Let $\Phi$ be a non-constant entire function of exponential type, with $\Phi(0)=1$.
\begin{enumerate}
\item[{\rm (a)}]\ If $\Phi$ is of subexponential growth, then $\Phi(D)$ has a hypercyclic algebra.

\item[{\rm (b)}]\. Suppose $\Phi$ is not of subexponential growth.
\begin{enumerate}
\item[1.]\  If $\Phi$ has no zeroes, then $\Phi(D)$ has no hypercyclic algebra.
\item[2.]\  If $\Phi$ is of the form $
\Phi(z)=e^{az} p(z)$
for some non-constant polynomial \[ p(z)=1+a_1z+a_2z^2+\dots +a_rz^r\] and some non-zero scalar $a$ with $a_1a^{-1}\in \C\setminus \R$ or with $2a_2\ne a_1^2$, then $\Phi(D)$ supports a hypercyclic algebra.
\item[3.]\ If $\Phi$ has infinitely many zeros and 
\[
\Phi(z) = e^{az} \prod_{n=1}^\infty E_p\left( \frac{z}{z_n} \right)
\]
denotes its Hadamard factorization, we have the following whenever $\sum_{n=1}^\infty  z_n^{-2}\ne 0$:
\begin{enumerate}
\item[{\rm (i)}]\ If $\sum_{n=1}^\infty |z_n|^{-1}$ converges, then $\Phi(D)$ has a hypercyclic algebra.

\item[{\rm (ii)}]\ If $\sum_{n=1}^\infty |z_n|^{-1}$ diverges and the symbol $a$ is non-zero then $\Phi(D)$ has a hypercyclic algebra.

\end{enumerate}

\end{enumerate}

\end{enumerate}

\end{theorem}

The proof of Theorem~\ref{T:Ts} is based on Theorem~\ref{T:2} and Theorem~\ref{T:ma^{m-1}} below.
\begin{theorem} \label{T:2}
Let $\Phi$ be entire of exponential type with $|\Phi (0)|=1$ and 
\[
\Phi''(0)\Phi(0)\ne \Phi'(0)^2.
\]
Suppose that $\Phi^{-1}(\mathbb{D)}$ contains arbitrarily long arithmetic progressions of the form
$\{ ja \}_{j=1}^m$, where $a\ne 0$ and $m\in \N$.
Then $\Phi(D)$ supports a hypercyclic algebra.
\end{theorem}

\begin{example}
Let $\Phi(z)=\mbox{cos}(z)$ or $\Phi(z)=\sin(z)+e^{-z}$. Then $\Phi(D)$ supports a hypercyclic algebra.
\end{example}
\begin{proof}
In either case we have $\Phi(0)=1$ and $\Phi''(0)\Phi(0)\ne (\Phi'(0))^2$, and $\Phi^{-1}(\D)$ has arbitrarily long arithmetic progressions of the desired form.
\end{proof}

Example~\ref{Ex:7} below complements Case $(b)3(ii)$ of Theorem~\ref{T:Ts}. 
\begin{example} \label{Ex:7}
Let $\Phi(z)=\frac{\sin(\pi z)}{\pi z}$. Then $\Phi(D)$ has a hypercyclic algebra.  
\end{example}
\begin{proof}
Here $\Phi(z)=\frac{\sin(\pi z)}{\pi z}=\prod_{0\ne k\in\Z} E_1(\frac{z}{k})$ 
has zero set $(z_n)=(-1,1,-2,2,\dots )$ (listed with multiplicity), which is  of divergence type, and 
$
n(r):=\# \{ n\in\mathbb{N}:   |z_n|\le r \} 
$ satisfies that $\limsup_{r\to\infty} \frac{n(r)}{r} =2\ne 0$, so $\Phi$ has order one and positive type, see \cite[Theorem~2.10.3(a)]{Boas}.
Also $\Phi^{-1}(\D)\supset \N$,  and  $\Phi''(0)\Phi(0)\ne (\Phi'(0))^2$ since $\Phi(0)=1$, $\Phi'(0)=0$ and $\Phi''(0)=\frac{\pi^2}{3}$.  So $\Phi(D)$ has a hypercyclic algebra by Theorem~\ref{T:2}.
\end{proof}



Theorem~\ref{T:ma^{m-1}} builds on the arguments used by Bayart in \cite[Theorem~3.1]{bayart}.

\begin{theorem} \label{T:ma^{m-1}}
Let $\Phi$ be of exponential type. Suppose there exist complex scalars $z_0$, $z_1$ with $z_0\in (0, z_1)$ so that
\begin{enumerate}
\item[{\rm (i)}]\ \ \ \ \ $|\Phi |< 1$ on $(0, z_0]$  and 

\item[{\rm (ii)}]\ $|\Phi(z_1)|>\mbox{max}\{ 1, e^{\tau_0 |z_1|}\}$, 
\end{enumerate}
where
\begin{equation} \label{eq:tau_0}
\tau_0=\tau_0(z_1)=\limsup_{r\to\infty} \frac{\log |\Phi (r\frac{z_1}{|z_1|})|}{r}.
\end{equation}
Then $\Phi(D)$ supports a hypercyclic algebra.
\end{theorem}

\begin{remark}\ \label{R:T:ma^{m-1}}

\begin{itemize}
\item[(a)]\
Every non-constant entire function $\Phi$ with $|\Phi(0)|=1$ satisfies condition $(i)$ in  a strong way. Indeed if $\Phi(z)=a_0+a_nz^n +\mbox{o}(z^n)$ with $a_n\ne 0$ and $n\ge 1$, then there exists 
a subdivision of the plane into 2n consecutive sectors $S_0,\dots, S_{2n-1}$, each of angle $\frac{\pi}{n}$, so that for each $k=1,\dots, n$ we have: On each ray lying in the interior of $S_{2k-1}$ the function $\Phi$ has initially modulus strictly smaller than one, and on each ray lying in the interior of $S_{2k}$ the function $\Phi$ initially has modulus strictly larger than one, see~\cite[pp 236--7]{gamelin}.
\item[(b)]\ Notice that
\[
\tau_0=\tau_0(z_1)=\begin{cases}
\ \ \ \  0 \ &\mbox{ if $\rho_{z_1}< 1$, }\\
h(\mbox{Arg}(z_1)) &\mbox{ if $\rho_{z_1}= 1$, where}
\end{cases}
\]
$h=h(\theta)=\limsup_{t\to\infty} t^{-1} \log |\Phi(t e^{i\theta})|$ is the Phragm\'{e}n-Lindel\"{o}f indicator function of $\Phi$ and where $\rho_{z_1}$ is the growth order of $\Phi$ in the direction of $z_1$, that is,\[\rho_{z_1}:=\limsup_{t\to\infty} \frac{\log\log |\Phi(t\frac{z_1}{|z_1|})|}{\log t}.\]
In particular, $\tau_0(z_1)=0$ whenever $\Phi$ is of subexponential growth, and
Theorem~\ref{T:ma^{m-1}} may be re-stated as follows:
\begin{quote}
{\em ``Let $\Phi$ be entire of exponential type, and let 
$h$ be its  Phragm\'{e}n-Lindel\"{o}f indicator function.
Suppose there exist $\theta\in [0,2\pi)$, and positive scalars  $r,R$ so that
\[
\begin{cases}
\ |\Phi (t e^{i\theta})| <1  &\mbox{ for $0<t<r$, and }\\
|\Phi (R e^{i\theta})| > \mbox{max}\{ 1, e^{h(\theta) R} \}.
\end{cases}
\]
Then $\Phi(D)$ supports a hypercyclic algebra.''}
\end{quote}
\end{itemize}
\end{remark}

Once a hypercyclic algebra is determined for a given operator, it is natural to ask how large such an algebra can be. While the hypercyclic algebras  for $D$ obtained in~\cite{shkarin2010on,bayart_matheron2009dynamics} were singly generated and thus not dense in $H(\mathbb{C})$, it was recently shown that $D$ and many other convolution operators support hypercyclic algebras that are both dense and infinitely generated~\cite{GroFal18,BP_algebrable,bernal-gonzalez_calderon-moreno}.  Moreover, for the case when $|\Phi(0)|<1$, the convolution operator $\Phi(D)$ supports a dense infinitely generated hypercyclic algebra if and only if it is not a scalar multiple of an exponential function, see~\cite[Theorem 6.3]{bayart}.  For the unimodular case we can say the following.

\begin{theorem} \label{T:IG}
Let $\Phi$ be a non-constant entire function of exponential type, with $|\Phi(0)|=1$. Then the generic element $f=(f_n)_{n=1}^\infty$ of $H(\mathbb{C})^\mathbb{N}$ freely generates a dense hypercyclic algebra for $\Phi(D)$ in each of the following cases:
\begin{enumerate}
\item[{\rm (a)}]\ $\Phi$ is of subexponential growth and  $\mbox{min}\{ n\in\mathbb{N}:  \Phi^{(n)}(0)\ne 0 \}$ is odd.

\item[{\rm (b)}]\ $\Phi$ is of the form $
\Phi(z)=e^{az+b} p(z)$
for some $a, b\in\mathbb{C}$ with $a\ne 0=\mbox{Re}(b)$ and some
non-constant polynomial $ p(z)=1+a_1z+a_2z^2+\dots +a_rz^r$ with $a_1a^{-1}\in \C\setminus \R$ or with both $2a_2\ne a_1^2$ and $a_1+a\ne 0$. 

\item[{\rm (c)}]\  $\Phi$ has Hadamard factorization
$
\Phi(z) = e^{az+b} \prod_{n=1}^\infty E_p\left( \frac{z}{z_n} \right)
$
satisfying \[ \sum_{n=1}^\infty  z_n^{-2}\ne 0, \   a\ne 0=\mbox{Re}(b),  \mbox{ and }\sum_{n=1}^\infty z_n^{-1}\ne a \mbox{ if } \sum_{n=1}^\infty |z_n|^{-1}<\infty.
\]
\end{enumerate}

\end{theorem}
For a statement to hold at a generic element of an $F$-space $X$ we mean that it holds at each element of a dense $G_\delta$ subset of $X$.
\begin{remark} In the language of {\em Lineability}~\cite{aron_bernal-gonzalez_pellegrino_seoane-sepulveda2015lineability}, Theorem~\ref{T:IG} states in particular that under each of the assumptions $(a)-(c)$ the set of hypercyclic vectors of the operator $\Phi(D)$ is 
{\em densely strongly algebrable} in the sense of Bartoszewicz and G{\l}{\c a}b \cite{barto}, see also~\cite[Section 2]{BP_algebrable}:
A subset $S$ of a commutative topological algebra $X$ is said to be densely strongly algebrable provided $S\cup\{ 0\}$ contains a dense subalgebra of $X$ that is induced by an infinite set of {\em free} generators.
\end{remark}

The paper is organized as follows. In Section~\ref{S:T:2} we show Theorem~\ref{T:2}. In Section~\ref{S:T:ma^{m-1}} we show Theorem~\ref{T:ma^{m-1}} and (all but Case (b)3(ii) of) Theorem~\ref{T:Ts} .
In Section~\ref{S:T:IG} we show Theorem~\ref{T:IG}, which require  
stronger versions of Theorem~\ref{T:2} and Theorem~\ref{T:ma^{m-1}} -Theorem~\ref{T:2bisbis} and Theorem~\ref{T:7bisbis}, respectively- whose proofs 
 are necessarily more technical. A particular case of the former, 
 Corollary~\ref{C:Phi'(0)}, establishes the pending Case (b)3(ii) of Theorem~\ref{T:Ts}.

\subsection{Notation and preliminaries} \label{Ss:Notation} 
We recall that the order $\rho$ of a constant entire function is defined as zero, and for a given non-constant $f\in H(\mathbb{C})$ it is given by
$
\rho:= \limsup_{r\to\infty}  \frac{\log \log M(r)}{\log r}$,
where $M(r)=M_f(r)=\mbox{sup}_{|z|\le r}|f(z)|$.  When $f$ has finite order $\rho$,  its type $\tau$ is given by
$\tau:=\limsup_{r\to\infty} r^{-\rho}\, \log M(r)$. 
Following Boas~\cite{Boas}, we say that $\Phi$ is of growth $(\rho, \tau)$ provided it is either of order less than $\rho$ or, if of order equal to $\rho$, it has type not exceeding $\tau$. We also say that $\Phi$ is of {\em subexponential growth} provided it is of growth $(1,0)$, and that $\Phi$ is of {\em exponential type} provided it is of growth $(1,\tau)$ for some $\tau<\infty$.   An entire function $\Phi(z)=\sum_{n=0}^\infty a_n z^n$ is of exponential type precisely when there exist constants $C,R\in (1,\infty)$ so that $|a_n|\le C \, R^n$ $(n\in\N )$, and this holds precisely when the operator $\Phi(D):H(\mathbb{C})\to H(\mathbb{C})$, $f\mapsto \sum_{n=0}^\infty a_n D^nf$, is well defined.

 According to the Hadamard Factorization Theorem~ \cite[Theorem~2.7.1]{Boas}, each  $f\in H(\mathbb{C})$ of finite order $\rho$   with  zero set $(z_n)$ listed with multiplicity and satisfying $0<|z_1|\le |z_2|\le \dots$ has unique factorization
\begin{equation} \label{eq:m1}
f(z)= e^{q(z)} P(z),
\end{equation}
for some polynomial $q(z)$ of degree not exceeding $\rho$ and infinite product
$
P(z)=\prod_{n=1}^\infty E_p(\frac{z}{z_n})
$  of genus
 \[p:=\mbox{min}\{ k\in\N_0:  \sum_{n=1}^\infty \frac{1}{|z_n|^k}=\infty \mbox{ and } \sum_{n=1}^\infty \frac{1}{|z_n|^{k+1}}<\infty \}.\]
Here $E_p$ denotes Weierstrass' canonical factor of order $p$, that is, $E_p(z)=1-z$ when $p=0$ and $E_p(z)=(1-z) e^{z+\frac{z^2}{2}+\cdots+\frac{z^p}{p}} $ when $p\in\N$, where $\mathbb{N}$ and $\mathbb{N}_0$ denote the sets of positive integers and of non-negative integers, respectively.
In particular, each $\Phi\in H(\mathbb{C})$  of exponential type with $|\Phi(0)|=1$ and which is not zero-free has unique factorization \[\Phi(z)=e^{az+b} P(z)\] for some  $a, b\in \mathbb{C}$ with $\mbox{Re}(b)=0$,
where
\[
P(z)=\begin{cases}
\prod_{n=1}^\infty (1-\frac{z}{z_n}) &\mbox{ if } \sum_{n=1}^\infty |z_n|^{-1}<\infty,\\
  \prod_{n=1}^\infty \{ (1-\frac{z}{z_n}) e^{\frac{z}{z_n}} \}&\mbox{ if } \sum_{n=1}^\infty |z_n|^{-1}=\infty,
\end{cases}
\]
and where $(z_n)$ is the (possibly finite) zero-set of $\Phi$, listed with multiplicity.

\section{Proof of Theorem~\ref{T:2}} \label{S:T:2}
We show Theorem~\ref{T:2} by means of the following sufficient criteria  for the existence of hypercyclic algebras.
\begin{lemma} \label{L:0}   {\bf {\rm (Bayart and Matheron~\cite[Remark~8.26]{bayart_matheron2009dynamics}})}
Let $T$ be an operator on a separable $F$-algebra $X$ so that for each triple $(U, V, W)$ of non-empty open subsets of $X$ with $0\in W$ and for each $m\in\N$ there exists $f\in U$ and $q\in\N$ so that
\begin{equation}  \label{eq:0*}
\begin{cases}
T^q(f^j) \in W \ \ \ \mbox{ for } 0\le j < m,\\
T^q (f^m)\in V.
\end{cases}
\end{equation}
Then the generic element of $X$ generates a hypercyclic algebra for $T$.
\end{lemma}
In order to apply Lemma~\ref{L:0} we first establish the next two lemmas.

\begin{lemma} \label{L:theta}
Let $A_1, A_2$ be complex scalars, where $A_2$ is non-zero. Then there exists $\theta\in\R$ so that
\begin{equation} \label{eq:Ltheta}
\begin{aligned}
0 &\le \mbox{Re}\left( A_1\, e^{i\theta} \right) \\
0 &< \mbox{Re}\left( A_2\, e^{i 2 \theta} \right).
\end{aligned}
\end{equation}
Moreover, if $A_1$ is also non-zero, then we may conclude both inequalities in $\eqref{eq:Ltheta}$ to be strict.
\end{lemma}
\begin{proof}
If $A_1=0$ the conclusion follows for example for $\theta = \frac{1}{2} \mbox{Arg} (\overline{A_2})$. If $A_1\ne 0$, we may select $\theta$ from the following table, according to which quadrants $A_1$ and $A_2$ belong to.

\begin{table}[h!]
\centering
\begin{tabular}{|c| c c c c  |} 
 \hline
 \diagbox{$A_2$}{$A_1$} & \ \ I  &\  II & III  &\ IV  \\ [0.5ex] 
 \hline
 I & \  \ 0  & \   $\pi^{+}$  &\  $\pi$ &\  $0$ \\ 
 II  & $(-\frac{\pi}{2})^+$ &  $(-\frac{\pi}{2})^+$ & \ $(\frac{\pi}{2})^+$ & $(\frac{\pi}{2})^+$ \\
 {\rm III} & $(-\frac{\pi}{2})^+$ &  $-\frac{\pi}{2}$ &\ $\frac{\pi}{2}$ & $\frac{\pi}{2}$ \\
 {\rm IV} &\ \ $0^{+}$  &  \  $\pi^{+} $ &\  $\pi^+$ &\  $0^{+}$ \\ 
 \hline
\end{tabular}
\label{table:1}
\end{table}
For example, if both $A_1$ and $A_2$ lie in the first quadrant we may choose $\theta = 0$. If $A_1$ and $A_2$ are in the second and first quadrant respectively, then we may choose $\theta = \pi^+$, that is, $\theta$ close to $\pi$ and slightly larger than $\pi$.
For the purpose of using this table we understand the first quadrant ${\rm QI}$ to be given by
\[
{\rm QI:}=\{ z\in \C: \ \mbox{Re}(z)>0 \ \mbox{ and } \mbox{Im}(z) \ge 0   \},
\]
and the remaining quadrants to be given similarly:
\[
\begin{aligned}
{\rm QII}&:=\{ z\in \C: \ \mbox{Re}(z)\le 0 \ \mbox{ and } \mbox{Im}(z) > 0   \}\\
{\rm QIII}&:=\{ z\in \C: \ \mbox{Re}(z)< 0 \ \mbox{ and } \mbox{Im}(z) \le 0   \}\\
{\rm QIV}&:=\{ z\in \C: \ \mbox{Re}(z)\ge0 \ \mbox{ and } \mbox{Im}(z) < 0   \}.
\end{aligned}
\]
The conclusion of the lemma now follows.
\end{proof}

\begin{lemma} \label{L:1}
Let $\Phi$ be entire 
and let $w_0$ so that 
\begin{equation} \label{eq:2.2bis}
\Phi(w_0)\ne 0 \mbox{ \ \  and \ \ }
\Phi''(w_0)\Phi(w_0)\ne \Phi'(w_0)^2. 
\end{equation}
Then we have:
\begin{enumerate}
\item[(a)]\ For each $\delta >0$ there exists $w_1$ in $D(w_0, \delta)$ so that the function
\[ 
[w_0,w_1]\to \R, \ z\mapsto \log|\Phi(z)|,
\] 
is strictly convex and $\log |\Phi (z)|$ strictly increases as $z$ moves from $w_0$ to $w_1$. That is, the function
\[
 t\mapsto \log |\Phi( (1-t) w_0+ t w_1)|,
\]
is both strictly convex and  strictly increasing on $[0, 1]$.

\item[(b)]\ If in addition to $\eqref{eq:2.2bis}$ we also assume $\Phi'(w_0)\ne 0$, then we may also conclude that
\[
 t\mapsto \log |\Phi( (1-t) w_0+ t w_1)|
\]
is both strictly increasing and strictly convex on $[-1,1]$.
\end{enumerate}
\end{lemma}

\begin{proof}
Reducing $\delta>0$ if necessary, we may assume that $\Phi$ is zero-free on $D(w_0, \delta)$, and thus that 
\begin{equation} \label{L1eq:0}
\Phi(z)=e^{h(z)}  \ \ \ z\in D(w_0, \delta)
\end{equation}
for some $h=h(z)$ holomorphic on $D(w_0,\delta)$. By our assumption we have
\begin{equation}   \label{L1eq:1}
h''(w_0) =\frac{\Phi''(w_0)\Phi(w_0) - \Phi'(w_0)^2 }{\Phi(w_0)^2} \ne 0,
\end{equation}
so we can write
\begin{equation} \label{eq:2.5bis}
h(z)=a_0+a_1 (z-w_0)+a_2 (z-w_0)^2 +\cdots \ \ \ \ \ (z\in D(w_0,\delta)),
\end{equation}
with $a_2\ne 0$.   By Lemma~\ref{L:theta} we may choose $\theta\in\R$ so that 
\begin{equation} \label{L1eq:2}
\mbox{Re}(a_1 e^{i\theta}) \ge 0 \ \mbox{ and } \ \mbox{Re}(a_2 e^{i 2 \theta})>0.
\end{equation}
Then the function
\begin{equation} \label{eq:2.10}
g:(-\delta, \delta)\to\R, \ g(t)=\mbox{Re}(h)(w_0+te^{i\theta})=\sum_{n=0}^\infty \mbox{Re}(a_n e^{in\theta}) t^n,
\end{equation}
is smooth and $g''(0)>0$ by our selection of $\theta$. So $g''>0$ on some $[-\rho, \rho]\subset (-\delta, \delta)$ with $\rho >0$ and since $g'(0)=\mbox{Re}(a_1 e^{i\theta })\ge 0$ there must exist $\eta$ in $(0,\rho)$ so that $g$ is strictly increasing on $[0,\eta]$ as well. Hence
\[
G:[0,1]\to \R, \ G(t):=g(\eta t)
\]
is both strictly convex and strictly increasing on $[0,1]$. But for $w_1:=w_0+\eta e^{i\theta}$ we have
\[
G(t)=\log |\Phi(w_0+ t \eta e^{i\theta})| = \log | \Phi ((1-t) w_0 + t w_1)| \ \ \ \ \ \ (t\in[0,1]),
\]
and $(a)$ holds.  To show $(b)$ we proceed similarly, noting that the new assumption $\Phi'(w_0)\ne 0$ gives that the expansion of $h$ in $\eqref{eq:2.5bis}$ also satisfies $a_1\ne 0$, and now the angle $\theta$ can be chosen so that both points $a_1 e^{i\theta}$ and $a_2 e^{i 2 \theta}$ in $\eqref{L1eq:2}$ 
have strictly positive real part, and thus
 the function $g$ in $\eqref{eq:2.10}$ has strictly positive first and second derivatives at the origin. Choosing $\rho\in (0,\delta)$ so that both $g'$ and $g''$ are strictly positive on $[-\rho, \rho]$, the conclusion $(b)$ now follows for $w_1:=w_0+ \rho e^{i\theta}$, where  $\log | \Phi ((1-t) w_0 + t w_1)| = G(t)=g(\rho t)$ for $t\in [-1,1]$ in this case.
\end{proof}   

We are now ready to show Theorem~\ref{T:2}.

\begin{proof}[Proof of Theorem~\ref{T:2}]
Let $U$, $V$, and $W$ be non-empty open subsets of $H(\C )$ with $0\in W$, and let $m\ge 2$ be given. By Lemma~\ref{L:0} it suffices to find some $(f,q)\in U\times\N$ so that
\[
\Phi(D)^q(f^m)\in V \ \ \mbox{ and } \ \ \Phi(D)^q(f^j)\in W \ (1\le j<m).
\]
Get $w$ a non-zero scalar so that $\{ w, 2w,\dots mw\} \subset \Phi^{-1}(\D)$, and chose $\delta>0$ small enough\begin{footnote}{To see why such $\delta>0$ exists: Let $g:\C^m\to \C$, $g(z_1,\dots, z_m):=\Phi(z_1+\dots+z_m)$. By our assumption, for each $s\in\{ 1,\dots, m\}$ we have
$g(\underset{s}{\underbrace{w,\dots,w}},0,\dots, 0)=\Phi(sw)\in\D$. The continuity of $g$ ensures the existence of some $\delta_s>0$ so that
\begin{equation}\label{eq:2.1}
\underset{s}{\underbrace{D(w,\delta_s)\times\cdots\times D(w,\delta_s)}}\times\underset{m-s}{\underbrace{D(0,\delta_s)\times\cdots\times D(0,\delta_s)}}\subset \Phi^{-1}(\D ).
\end{equation}
In particular, for each $w_1,\dots, w_s\in D(w,\delta_s)$ and $v_1,\dots, v_d\in D(0,\delta_s)$ with $s+d\le m$ we have
\[
\Phi(w_1+\dots+w_s+v_1+\dots+v_d)=g(w_1,\dots ,w_s, v_1,\dots ,v_d,0,\dots, 0)\in \D.
\]
Then $\delta:=\mbox{min}\{ d_s: \ 1\le s\le m\}$ does the job.
} \end{footnote} so that for each $s\in\{ 1,\dots, m\}$ and  $w_1,\dots, w_s\in D(w,\delta)$ and for each $d\in \{ 0,\dots, m-s\}$ and $v_1,\dots, v_d\in D(0,\delta)$ we have
\begin{equation}\label{eq:2.2}
w_1+\cdots+w_s+v_1+\cdots+v_d\in \Phi^{-1}(\D),
\end{equation}
where the expression $v_1+\cdots+v_d$ is defined as zero whenever $d=0$.  Now, by Lemma~\ref{L:1} there exists a non-zero scalar $w^*$ in $D(0,\delta)$ so that the map
\[
[0, w^*]\to \R, \ \ z\mapsto \log |\Phi(z)|,
\]
is strictly convex, and it increases with $|z|$ whenever $z\in [0,w_0]$, for some $w_0\in (0,w^*)$.  Now, since the sets $D(w,\delta)$ and $[\frac{w_0}{2}, w_0]$ have accumulation points in $\C$ there exist $p\in\N$,  $\alpha_1,\dots, \alpha_p\in D(w,\delta)$,  and $\beta_1,\dots, \beta_p\in [\frac{w_0}{2}, w_0]$ and scalars $a_i, b_i\in\C\setminus\{ 0 \}$,   $i\in I_p:=\{ 1,\dots, p \}$, so that
\[
(A,B):=\left(\sum_{i\in I_p} a_i e^{\alpha_i z}, \sum_{i\in I_p} b_i e^{\beta_i z} \right)\in U\times V.
\]
Next, set $R_N:=\sum_{i=1}^p c_i e^{\lambda_i z}$, where $(m\lambda_1,\dots, m\lambda_p)=(\beta_1,\dots ,\beta_p)$ and where $(c_1,\dots, c_p)=(c_1(N),\dots, c_p(N))$ is a solution of the system
\begin{equation}\label{eq:2.3}
b_i=c_i^m \Phi(m\lambda_i)^N \ \ \ (i\in I_p).
\end{equation}
Notice that for each $i\in I_p$ we have $c_i=c_i(N)\underset{N\to\infty}{\to} 0$, thanks to having $|\Phi (m\lambda_i)|=|\Phi(\beta_i)|>1$ and $b_i\ne 0$. So $f=f_N:=(A+R_N)$ belongs to $U$ whenever $N$ is large enough. 

Now, let $\mathbb{E}_p:=\{ e_1,\dots, e_p \}$ denote the canonical basis of $\C^p$, and for any $x=(x_1,\dots ,x_p)\in \N_0^p$ let $|x|:=\sum_{i\in I_p} x_i$.  Also, given $x,y\in \N_0^p$ we let $xy:=\sum_{i\in I_p} x_iy_i$. Finally,
for each $n\in \N$ we let $\mathcal{L}_n$ denote the set of tuples $(u, v)$ in $\N_0^p\times \N_0^p$ for which $|u|+|v|=n$.
Our selection of $c_i$'s and $\lambda_i$'s ensures that
\[
\begin{cases}
\Phi(D)^N(f^m)=B+ \sum_{(u,v)\in \mathcal{L}_m^*} S(u,v,N) e^{(\alpha u +\lambda v)z} \\
\Phi(D)^N(f^n)= \sum_{(u,v)\in \mathcal{L}_n} S(u,v, N) e^{(\alpha u +\lambda v)z} \ \ (1\le n <m)
\end{cases}
\]
where
\[
\mathcal{L}_m^*:=\mathcal{L}_m\setminus \{ (u,v)\in \mathcal{L}_m: \ |u|=0 \mbox{ and } v\in mE_p \}
\]
 and where for each fixed $(u, v)\in \mathcal{L}:=\mathcal{L}_m^*\cup \mathcal{L}_1 \cup\dots \cup \mathcal{L}_{m-1}$ there exists a scalar $\gamma =\gamma_{u,v}$ that does not depend\begin{footnote}{Indeed, the multinomial theorem gives
 $\gamma_{u,v}=(|u|+|v|)! \prod_{i=1}^p \frac{a_i^{u_i}}{u_i!  v_i!}$ }\end{footnote} on $N$ so that
 \begin{equation}  \label{eq:3amSF}
 S(u,v, N)= \gamma \ \Phi (\alpha u+\lambda v)^N \prod_{i\in I_p} c_i^{v_i}
 \end{equation}
Pending is to show that for each $(u,v)\in \mathcal{L}$ we have
\[
S(u,v, N)\underset{N\to\infty}{\to} 0.
\]
 Now, let $(u,v)\in \mathcal{L}$ be fixed.  By $\eqref{eq:2.3}$ and $\eqref{eq:3amSF}$ we have
 \begin{equation}
 |S(u,v,N)|=\bigo(\theta_{u,v}^N) \ \ \mbox{ as $N\to\infty$, }
 \end{equation}
 where $\theta_{u,v}=\frac{ |\Phi(\alpha u+ \lambda v)| }{ \prod_{i\in I_p} |\Phi (m\lambda_i)|^{v_i/m}}$.   So it suffices to verify that $\theta_{u,v}\in [0,1)$.   We have three cases:
 
 \begin{enumerate}
 \item[{\rm Case 1:}]\ $(u, v)\in \mathcal{L}_m^*$ with  $|u|=0$. Here $|v|=m$ and $v\notin m\mathbb{E}_p$, so $\lambda v$ is a non-trivial convex combination of $m\lambda_1, \dots, m\lambda_p$.  Hence
 \[
 \theta_{u,v} = \frac{ |\Phi(\lambda v)| }{\prod_{i\in I_p} | \Phi(m\lambda_i) |^{v_i/m}} < 1
 \]
 thanks to the map $z\mapsto \log |\Phi(z)|$ being strictly convex on $[0, w^*]$ and our selection of the $\lambda_i$'s.
 
 \item[{\rm Case 2:}] $1\le |u|\le m$ and $|v|\le m-|u|$.  By $\eqref{eq:2.2}$ we have $\alpha u+\lambda v$ in $\Phi^{-1}(\D )$. Also, since $|\Phi(0)|=1$ and since $z\mapsto \log |\Phi(z)|$ strictly increases with $|z|$ whenever $z\in [0, w_0]$, we have that $|\Phi (m\lambda_i)|>1$ for each $i\in I_p$. So $\theta_{u,v}\in [0,1)$.

 \item[{\rm Case 3:}]\  $|u|=0$ and $1\le |v|\le m-1$.    Notice that in this case
 \[
 \alpha u+ \lambda v = 0 (1-\frac{|v|}{m}) + m\lambda_1\frac{v_1}{m}+\dots + m\lambda_p \frac{v_p}{m}
 \]
 is a non-trivial convex combination of the points $0, m\lambda_1,\dots ,m\lambda_p$ lying on $[0, w^*]$. So the strict convexity of $z\mapsto \log|\Phi(z)|$ on $[0, w^*]$ and the fact that $|\Phi(0)|=1$ give
 \[
  \begin{aligned}
 |\Phi(\alpha u+\lambda v)|&< |\Phi(0)|^{1-\frac{|v|}{m}} \prod_{i\in I_p} |\Phi(m\lambda_i)|^{v_i/m}\\
 & = \prod_{i\in I_p} |\Phi(m\lambda_i)|^{v_i/m}, 
  \end{aligned}
 \]
 and thus that $\theta_{u,v}\in [0,1)$. 
 \end{enumerate}
The proof of Theorem~\ref{T:2}. is now complete.                                                          
\end{proof}

\section{Proofs of Theorem~\ref{T:ma^{m-1}} and Theorem~\ref{T:Ts}.}   \label{S:T:ma^{m-1}}

We use the following result by Bayart to show Theorem~\ref{T:ma^{m-1}}. 
\begin{theorem} {\bf (Bayart)} \cite[pp 3452--3454]{bayart} \label{T:bayart}
Let $\Phi$ be entire and of exponential type. Suppose that for each $m\ge 2$ there exist scalars $z_0$  and $w_0$ with $z_0\in (0, w_0)$ so that
\begin{enumerate}
\item[{\rm (a)}]\  $|\Phi (w_0)|>1$ and  $|\Phi|<1$ on $(0,z_0]$, 

\item[{\rm (b)}]\ $|\Phi(dw_0)| < |\Phi(w_0)|^d$ for each $d\in\{ 2,3,\dots, m\}$, and

\item[{\rm (c)}]\ The function $t\mapsto |\Phi (w_0+tz_0)|$ is strictly increasing on $[0, \eta)$, for some $\eta>0$.
\end{enumerate}

Then $\Phi(D)$ supports a hypercyclic algebra on $H(\mathbb{C})$.
\end{theorem}

We also use the following elementary fact: 
\begin{remark} \label{R:1}
Given $f:[a,b]\to \R$ continuously differentiable and $y\in (f(a), f(b))$, there exists $c\in (a,b)$ so that $y<f$ on $[c, b]$ and $f'(c)>0$.
\end{remark}

\begin{proof}[Proof of Theorem~\ref{T:ma^{m-1}}]
Fix $m\ge 2$. By  (i) and Theorem~\ref{T:bayart}, it suffices to show that there exists some $w_0\in \{ tz_0: t>1\}$ so that
\begin{enumerate}
\item[(a')]\  $|\Phi (w_0)|>1$,  

\item[(b')]\ $|\Phi(dw_0)| < |\Phi(w_0)|^d$ for each $d\in\{ 2,3,\dots, m\}$, and

\item[(c')]\ The function $t\mapsto |\Phi (w_0+tz_0)|$ is strictly increasing on an interval\footnote{While we only need $[0,\eta)$ in (c') for this proof, we later use in the proof of Theorem~\ref{T:7bisbis} this stronger conclusion of having $(-\eta, \eta)$ in (c').}
$(-\eta, \eta)$, for some $\eta>0$.
\end{enumerate}
By means of contradiction, suppose no such $w_0$ satisfies $(a')-(c')$.
Let $t^*>1$ so that $z_1=t^*z_0$. Also, consider the smooth function $\psi:\R\to\R$, $\psi(t)=|\Phi(tz_0)|^2$. By $(ii)$ we may pick $0<\epsilon <\psi(t^*)-\mbox{max}\{ 1, e^{2\tau_0 |z_1|} \}$. Also, let $\frac{1}{4} > \epsilon_1>\epsilon_2>\dots >0$ so that
\[
\prod_{k=1}^\infty (1-\epsilon_k) > 1-\epsilon .
\]
By $(ii)$ we have 
\[
\psi(1)< 1 \le \mbox{max}\{ 1, e^{2\tau_0 |z_1|} \} < \psi (t^*).
\]
So by Remark~\ref{R:1} there exists $t_0\in (1, t^*)$ so that
\begin{equation}\label{eq:zero}
\begin{aligned}
\mbox{max}\{ 1, e^{2\tau_0 |z_1|}\} < \psi(t^*) -\epsilon &<\psi(t_0). \\
0&<\psi'(t_0).
\end{aligned}
\end{equation}
In particular, $1<\psi(t_0)=|\Phi(t_0z_0)|^2$ and $\psi$ is strictly increasing on $(t_0-\eta, t_0+\eta)$ for some $\eta>0$; in other words the function $t\mapsto |\Phi((t+t_0)z_0)| = |\Phi(t_0z_0+tz_0)|$ is strictly increasing on $(-\eta,\eta)$.  That is,  $w_0:=t_0z_0$ satisfies $(a)$ and $(c)$ and hence it must fail $(b)$ so for some $d_1\in \{ 2,\dots, m \}$ we have 
\[
\psi(d_1t_0) \ge \psi(t_0)^{d_1} > \psi(t_0)^{d_1 (1-\epsilon_1)} >\psi(t_0).
\]
By Remark~\ref{R:1}, there exists $t_1\in (t_0, d_1t_0)$ so that
\begin{equation}\label{eq:one}
\begin{aligned}
\psi(t_0)^{d_1 (1-\epsilon_1)} &< \psi(t_1) \\
0&< \psi'(t_1).
\end{aligned}
\end{equation}
Let $r_1\in (1,d_1)$ so that $t_1=r_1t_0$. As before $w_0:=t_1z_0$ satisfies $(a)$ and $(c)$, so there exists $d_2\in \{ 2,\dots, m \}$ so that
\[
\psi(d_2t_1) \ge \psi(t_1)^{d_2} > \psi(t_1)^{d_2 (1-\epsilon_2)} >\psi(t_0)^{d_1(1-\epsilon_1) d_2 (1-\epsilon_2)} > \psi(t_0).
\]
By Remark~\ref{R:1} there exists $t_2\in (t_1, d_2t_1)$ so that 
\begin{equation}\label{eq:two}
\begin{aligned}
\psi(t_1)^{d_2 (1-\epsilon_2)} &< \psi(t_2) \\
0&< \psi'(t_2).
\end{aligned}
\end{equation}
We let $r_2\in (1, d_2)$ so that $t_2=r_2t_1$, and note again that $w_0:=t_2z_0$ satisfies $(a)$ and $(c)$ to get $d_3\in \{ 2,\dots, m\}$ with 
\[
\psi(d_3t_2)\ge \psi(t_2)^{d_3}> \psi(t_2)^{d_3 (1-\epsilon_3)} >\psi(t_2).
\]
Continuing this process inductively, we obtain sequences $(t_n)_{n=1}^\infty$, $(d_n)_{n=1}^\infty$, and $(r_n)_{n=1}^\infty$ so that for each $n\ge 1$ we have
\begin{equation} \label{eq:astast}
\begin{aligned}
d_n&\in \{ 2,\dots, m \} \\
 r_n&\in (1, d_n) \\
t_n&=r_n t_{n-1} \\    
\psi(t_n)&> \psi(t_{n-1})^{d_n (1-\epsilon_n)} > \psi(t_0)^{(3/2)^n}\\
 \psi'(t_n)&>0.
\end{aligned}
\end{equation}
In particular, $\psi(t_n)\to \infty$ since $\psi(t_0)>1$, and thus $t_n{\to}\infty$. Also, by $\eqref{eq:astast}$ we have  $t_n=r_nt_{n-1}=\dots =(r_nr_{n-1}\cdots r_1)t_0$ and
\[
d_1\cdots d_n \prod_{k=1}^n(1-\epsilon_k)  >  d_1\cdots d_n(1-\epsilon)  >  r_1\cdots r_n(1-\epsilon)  = (1-\epsilon) \frac{t_n}{t_0},
\]
giving
\[
\begin{aligned}
\psi(t_n)&> \psi(t_{n-1})^{d_n(1-\epsilon_n)}\\
&> \psi(t_{n-2})^{d_n(1-\epsilon_n) d_{n-1} (1-\epsilon_{n-1})} \\
&\ \vdots \\
&>\psi(t_0)^{ d_1\cdots d_n \prod_{k=1}^n(1-\epsilon_k) } \\
&> \psi(t_0)^{(1-\epsilon)\frac{t_n}{t_0}}.
\end{aligned}
\]
So for each $n\ge 1$
\[
\begin{aligned}
|\Phi(t_nz_0)| & > |\Phi(t_0z_0)|^{(1-\epsilon)\frac{t_n}{t_0}} \\
&= e^{(1-\epsilon)\log|\Phi(t_0z_0)| \frac{t_n}{t_0}}\\
&> e^{(1-\epsilon)\log (( |\Phi(t^*z_0)|^2-\epsilon )^\frac{1}{2}) \frac{t_n}{t_0}     \ \ \mbox{(by }\eqref{eq:zero})}\\
&> e^{(1-\epsilon)\frac{\log (( |\Phi(z_1)|^2-\epsilon )^\frac{1}{2})}{|z_1|}    t_n|z_0| }\\
&= e^{(1-\epsilon)\frac{\log (( |\Phi(t^*z_0)|^2-\epsilon )^\frac{1}{2})}{t^*|z_0|}    t_n|z_0| }.
\end{aligned}
\]
Hence for each small $\epsilon>0$ we have
\[
\tau_0\ge (1-\epsilon) \frac{\log( (|\Phi(z_1)|^2-\epsilon)^\frac{1}{2})}{|z_1|}
\]
forcing $\tau_0 \ge \frac{\log|\Phi(z_1)|}{|z_1|}$, contradicting our assumption $\eqref{eq:tau_0}$.
\end{proof}

\begin{corollary} \label{C:7.1}
Let $\Phi(z)=e^{az} \varphi (z)$, where $\varphi$ is non-constant of growth $(\frac{1}{2}, 0)$ with $|\varphi(0)|=1$, and 
$\mbox{Im}\left( a\, \varphi(0)\, {\overline{\varphi'(0)}} \right)\ne 0$.
Then $\Phi(D)$ supports a hypercyclic algebra. 
\end{corollary}
To show this we recall the following fact.

\begin{lemma}  \label{L:CD}   \cite[Lemma~5]{BCP2018}
Let $\Phi \in H(\C )$ be of exponential type, and consider the composition operator $C_\varphi :H(\C )\to H(\C )$, $f\mapsto f\circ\varphi$, where $\varphi :\C \to \C$, $\varphi (z)=a z$, is a homothety on the plane with $0\ne a\in\C$. Then  $\Phi_a :=C_\varphi (\Phi)$ is of exponential type and 
\[
C_\varphi (HC(\Phi_a (D)))=HC (\Phi (D)).
\]
In particular, the algebra isomorphism $C_\varphi :H(\C )\to H(\C )$ maps hypercyclic algebras of $\Phi_a (D)$ onto hypercyclic algebras of $\Phi (D)$.
\end{lemma}

\begin{proof}[Proof of Corollary~\ref{C:7.1}]
By Lemma~\ref{L:CD} the convolution operator $\Phi(D)$ has the same set of hypercyclic vectors as the one induced by 
 \[
 \Phi_{a^{-1}}(z)=\Phi(\frac{z}{a}) = e^z \varphi_{a^{-1}}(z),
 \]
 where $\varphi_{a^{-1}}(z)=\varphi (\frac{z}{a} )$.    Notice 
  that $\varphi_{a^{-1}}$ is also of subexponential growth and that it is of growth $(\frac{1}{2}, 0)$ whenever $\varphi$ is, and that $|\Phi_{a^{-1}}|=|\varphi_{a^{-1}}|$ on the imaginary axis. It suffices to show $\Phi_{a^{-1}}$ satisfies the assumptions of Theorem~\ref{T:ma^{m-1}}. Letting
$g(t):=|\Phi_{a^{-1}}(it)|^2=|\varphi (\frac{it}{a})|^2$, we have $
g'(0)
=\frac{2}{|a|^2} \mbox{Im}\left( a\, \varphi(0) \overline{ \varphi'(0)} \right)  \ne 0
$
and thus there exists some $\eta>0$ so that
$\Phi_{a^{-1}}$ has modulus smaller than one on $(0, \eta i)$ if $g'(0)<0$ or on $(-\eta i, 0)$ if $g'(0)>0$.  Also
$\varphi_{a^{-1}}$  is unbounded on any half line, since it is of growth $(\frac{1}{2}, 0)$ \cite[Theorem 3.1.5]{Boas}. 
\end{proof}

\begin{corollary} \label{C:7.4}
Let $\Phi (z) = e^{az}  \prod_{n=1}^\infty (1-\frac{z}{z_n})$, where $a$ and $z_n$ $(n\in\N)$ are non-zero complex scalars satisfying
\begin{enumerate}
\item[{\rm (i)}]\ $\sum_{n=1}^\infty \frac{1}{|z_n|} <\infty$.
\item[{\rm (ii)}]\  $\sum_{n=1}^\infty \, \frac{1}{z_n^2} \ne 0$. 
\end{enumerate}
Then $\Phi(D)$ has a hypercyclic algebra.
\end{corollary}

\begin{proof}
Condition $(i)$ ensures that $P(z):=\prod_{n=1}^\infty (1-\frac{z}{z_n})$ converges absolutely and locally uniformly on $\C$, and that $P$ is of growth $(1,0)$ \cite[Lemma 2.10.13]{Boas}.  Hence $\Phi$ is non-constant of exponential type.
By Lemma~\ref{L:CD}, replacing $(z_n)$ by $(a z_n)$ if necessary, without loss of generality we may assume that $a=1$. We show that  $\Phi (z)=e^{z} P(z)$ satisfies the assumptions of Theorem~\ref{T:2}. Notice first that $\Phi''(0)\Phi(0)\ne \Phi'(0)^2$ if and only \begin{equation}\label{eq:C7.41}
P''(0)P(0)\ne P'(0)^2,\end{equation}
 and that $P(0)=1$, $P'(0)=-\sum_{n=1}^\infty \frac{1}{z_n}$, and $P''(0)=(-\sum_{n=1}^\infty \frac{1}{z_n})^2+\sum_{n=1}^\infty \, \frac{1}{z_n^2}$.  So $\eqref{eq:C7.41}$ holds precisely when $(ii)$ does. 
Finally, since $P$ is of growth $(1, 0)$ we have $\lim_{x\to -\infty} e^{x} P(x)=0$, so $\Phi^{-1}(\mathbb{D})$ contains arithmetic progressions of the desired form.
\end{proof} 


The case $a=0$ of Corollary~\ref{C:7.4} is covered by Corollary~\ref{C:7.3}.
\begin{corollary} \label{C:7.3}
Let $\Phi \in H(\C)$ be non-constant and of subexponential growth, with $|\Phi (0)|=1$.  Then $\Phi (D)$ has a hypercyclic algebra.
\end{corollary}
\begin{proof}
By Theorem~\ref{T:ma^{m-1}} it suffices to show that there exist $z_1\in\C$ and $0<\delta <1$ so that 
\begin{equation}
\begin{cases} \label{eq:AST2}
\mbox{ \ \ \ \ \  $|\Phi |<1$ on $(0, \delta z_1]$ }\\
\mbox{ $|\Phi(z_1)|>\mbox{max}\{ 1, e^{\tau_0 |z_1|}\}$, }
\end{cases}
\end{equation}
where $\tau_0=\tau_0(z_1)=\limsup_{r\to\infty} r^{-1} \, \log |\Phi (r \frac{z_1}{|z_1|})|$.
Since $\Phi\in H(\C)$ is non-constant, we may write 
\[
\Phi(z)=\Phi(0)+a_N z^N + o(Z^N) \ \ \ (z\to 0),
\]
with $N\ge 1$ and $a_N\ne 0$.  Notice also that for each $A\in \C\setminus\{ 0 \}$ and each $\theta\in\R$ we have that $\Phi$ satisfies $\eqref{eq:AST2}$ for $(z_1,\delta, \tau)$ if and only if 
\[
\widetilde{\Phi}(z)=e^{i\theta} \Phi(Az)
\]
satisfies $\eqref{eq:AST2}$ for $(\widetilde{z}_1, \widetilde{\delta}, \widetilde{\tau})=(\frac{z_1}{A}, \delta, |A|\tau)$.
So taking $\theta\in\R$ and $A\in\mathbb{C}$ of modulus one so that
\[
\begin{aligned}
e^{i\theta} \Phi(0)&=1\\
e^{i\theta} a_N A^N&<0,
\end{aligned}
\]
without loss of generality we may assume that
\begin{equation}
\label{eq:ASTAST}
\Phi(z)=1-az^N+o(z^N) \ \ (z\to 0)
\end{equation}
where $a>0$ and $N\ge 1$.  Since $\Phi$ is of growth $(1,0)$ it can't be bounded on any open half-plane. So there exists some $z_0\in\C$ with 
\[
\mbox{Re}(z_0)>0 \mbox{ and } |\Phi(z_0)|>1.
\] 
Now, for each $0\le j<N$ let $\Phi_j(z):= \Phi(w^jz)$, where $w:=e^{i\frac{2\pi}{N}}$. Notice that each $\Phi_j$ is of growth $(1,0)$ and satisfies $\eqref{eq:ASTAST}$.   Hence replacing $\Phi$ by $\Phi_j$ for some $0\le j<N$ if necessary, without loss of generality we may assume that $\mbox{Arg}(z_0)\in (-\frac{\pi}{2N}, \frac{\pi}{2N})$. Hence for $t\in\R$ we have
\begin{equation}
\begin{aligned}\label{eq:dibujo}
\Phi(tz_0)&= 1-at^Nz_0^N+g(t)\\
&=1-t^N \left( a z_0^N -\frac{g(t)}{t^N}\right),
\end{aligned}
\end{equation}
where $g(t)=o(t^N)$ as $t\to 0$. Now, let $K$ be a closed disc centered at $a z_0^N$ and contained in $\{ z\in\C: \ \mbox{Re}(z)>0 \}$.  Then there exists $\delta>0$ so that for any $0<s<\delta$ the set $sK$ is contained in the open disc $D(1,1)$. Hence by $\eqref{eq:dibujo}$ for each $0<t<\delta$ we have $|\Phi(tz_0)|<1$.  So $z_1:=z_0$ works, because by Remark~\ref{R:T:ma^{m-1}}(b)  we have  $\tau_0 =0$, as $\Phi$ is of subexponential growth.\end{proof}

\begin{lemma} \label{L:4T3.1}
Let $\Phi(z)=p(z)e^z$, where $p(z)=1+a_1z+a_2z^2+\dots +a_rz^r$, with $r\ge 1$ and $a_r\ne 0$.  Suppose that  $\mbox{Im}(a_1)\ne 0$ or that $ 2 a_2\ne  a_1^2$.  Then $\Phi(D)$ supports a hypercyclic algebra.    
\end{lemma}
\begin{proof}
The case $\mbox{Im}(a_1)\ne 0$ is immediate from Corollary~\ref{C:7.1}. So assume that $a_1\in\R$. We show that $\Phi$ satisfies the assumptions of Theorem~\ref{T:2}.  Since $|\Phi (x)|\to 0$ as $x\to -\infty$,  
the set $\Phi^{-1}(\D)$ contains an interval of the form $(-\infty, x]$, and in particular the desired arithmetic progressions. But $\Phi(0)=1$ and $
\Phi''(0)\Phi(0)-\Phi'(0)^2 =2a_2-a_1^2\ne 0$. \end{proof}

\begin{proof}[Proof of Theorem~\ref{T:Ts}]
Conclusion (a) has been established in Corollary~\ref{C:7.3}. So assume $\Phi$ is not of subexponential growth, that is, it has growth order one and positive  finite type. If $\Phi$ is zero-free it is a scalar multiple of an exponential and Conclusion (b1) holds, see \cite{aron_conejero_peris_seoane-sepulveda2007powers}. Conclusion (b2) follows from
 Lemma~\ref{L:CD} and Lemma~\ref{L:4T3.1}, noting that the scalar $a$ must be non-zero in this case as $\Phi$ is of order one. 
 To show (b3), consider the Hadamard factorization
 \[
 \Phi(z)=e^{az} P(z),
 \]
 of $\Phi$, where $P(z)=\prod_{n=1}^\infty E_p(\frac{z}{z_n})$ is the canonical product of genus $p$ of the zero-set $(z_n)$ of $\Phi$. Notice that $p$ must be zero or one, since $\Phi$ has order one.
 Case (i), when $\sum_{n=1}^\infty |z_n|^{-1}<\infty$, has been established in Corollary~\ref{C:7.4} as $p=0$ in this case. Case $(ii)$, where $\sum_{n=1}^\infty |z_n|^{-1}=\infty$, implies we have $p=1$ and hence
 \[
 P(z)=\prod_{n=1}^{\infty} \{ (1-\frac{z}{z_n})\, e^\frac{z}{z_n}\},
 \]
so $P(0)=1$, $P'(0)=0$, and $P''(0)=\sum_{n=1}^\infty z_n^{-2}$.   So by our assumption we have $P''(0)P(0)\ne P'(0)^2$, and equivalently that $\Phi''(0)\Phi(0)\ne \Phi'(0)^2$. Also, $\Phi'(0)=P'(0)+aP(0)=a\ne 0$ by our assumption on $a$, and the conclusion follows from Corollary~\ref{C:Phi'(0)}.
\end{proof}

\begin{corollary}
Let $(z_n)$ be a sequence of complex scalars with $0<|z_1|\le |z_2|\le \dots$ and satisfying
\begin{enumerate}
\item[{\rm (i)}]\ $\sum_{n=1}^\infty |z_n|^{-1}$ diverges, and $\sum_{n=1}^\infty |z_n|^{-p}$ converges for each $p>1$.
\item[{\rm (ii)}]\  $\limsup_{r\to\infty}  r^{-1} \mbox{max}\{j\in\N: \ |z_j|\le r\} <\infty$, 
\item[{\rm (iii)}]\ The set  $\{ \sum_{|z_n|\le r} z_n^{-1} :\ r >0 \}$ is bounded, and
\item[{\rm (iv)}]\ $\sum_{n=1}^\infty z_n^{-2} \ne 0$.
\end{enumerate}
Then for each 
\[
\Phi(z)=e^{az} \, \prod_{n=1}^\infty \{ (1-\frac{z}{z_n}) e^\frac{z}{z_n} \}
\]
with $a\ne 0$, the operator $\Phi(D)$ supports a hypercyclic algebra. 
\end{corollary}
\begin{proof}
By $(i)$ the convergence exponent of the zeros of $f$ is $\rho_1=1$, and 
the 
canonical product $\prod_{n=1}^\infty \{ (1-\frac{z}{z_n}) e^\frac{z}{z_n} \}$ has order equal to  the convergence exponent of its zeros \cite[Theorem~2.6.5]{Boas}, which is equal to $1$, by $(i)$. So $\Phi$ is of order at most $1$, and by $(ii)$ and $(iii)$ we may conclude $\Phi$ is of exponential type \cite[Theorem 2.9.5 and Theorem 2.10.1]{Boas}. The conclusion follows from Theorem~\ref{T:Ts}. 
\end{proof}

The following consequence of Theorem~\ref{T:Ts} complements \cite[Corollary 14]{BCP2018}.
\begin{corollary}
Let $p=p(z)$ be a non-constant polynomial with $|p(0)|\le 1$. Then the differentiation operator  $p(\frac{d}{dx})$
acting on $\mathbb{C}^{\infty}(\mathbb{R}, \mathbb{C})$ supports a hypercyclic algebra.
\end{corollary}
\begin{proof} The operator $p(\frac{d}{dx})$ is quasi-conjugate to $p(D):H(\mathbb{C})\to H(\mathbb{C})$ via a multiplicative operator, so it suffices to verify that $p(D)$ supports a hypercyclic algebra. The conclusion now follows from Theorem~\ref{T:1.1} when $|p(0)|<1$ and from the Le\'{o}n-M\"{u}ller Theorem and Case (a)  of Theorem~\ref{T:Ts} when $|p(0)|=1$.
\end{proof}

\section{Proof of Theorem~\ref{T:IG}.} \label{S:T:IG}

We note that the same arguments used in the proofs of Theorem~\ref{T:Ts}, Theorem~\ref{T:2} and Theorem~\ref{T:ma^{m-1}} give general eigenvalue criteria for the existence of hypercyclic algebras. This general formulation was used by Bayart~\cite{bayart} to provide hypercyclic algebras for operators on the space $\ell_1(\mathbb{N})$ endowed with the convolution product, for instance. A criterion corresponding to Theorem~\ref{T:2}, for example, may be stated as follows.

\begin{theorem}\label{T:2bis}
Let $T$ be an operator on a separable infinite dimensional $F$-algebra $X$. Suppose there exist a function $E:\mathbb{C}\to X$ and an entire function $\phi:\mathbb{C}\to\mathbb{C}$ satisfying the following:
\begin{enumerate}
\item[{\rm (a)}] \ For each $\lambda\in\mathbb{C}$, $TE(\lambda)=\phi(\lambda) E(\lambda)$,

\item[{\rm (b)}]\ For each $\lambda, \mu\in\mathbb{C}$, $E(\lambda )E(\mu)= E(\lambda+\mu)$,

\item[{\rm (c)}]\ For each subset $\Lambda$ of $\mathbb{C}$ supporting an accumulation point in $\mathbb{C}$, the set $\{ E(\lambda):\ \lambda\in\Lambda \}$ has dense linear span in $X$,

\item[{\rm (d)}]\ The function $\phi$ satisfies $|\phi (0)|=1$, and $\phi''(0)\phi (0) \ne \phi'(0)^2$, and

\item[{\rm (e)}]\ For each positive integer $m$ there exists a non-zero scalar $a=a(m)$ so that $\{ a, 2a, \dots, ma \}\subset \phi^{-1}(\mathbb{D})$.

\end{enumerate}
Then the singly generated algebra induced by a generic element $f$ of $X$ is a hypercyclic algebra for $T$.
\end{theorem}

We adopt this point of view when providing with Theorem~\ref{T:2bisbis} and Theorem~\ref{T:7bisbis} below general eigenvalue criteria for the existence of dense hypercyclic algebras induced by infinitely many free generators, and use such criteria to establish Theorem~\ref{T:IG}.

\begin{theorem}\label{T:2bisbis}
Let $X$ be a separable commutative $F$-algebra that supports a dense freely generated subalgebra, and let $T$ be an operator on $X$. Suppose there exist a function $E:\mathbb{C}\to X$ and an entire function $\phi:\mathbb{C}\to\mathbb{C}$ satisfying the following:
\begin{enumerate}
\item[{\rm (a)}] \ For each $\lambda\in\mathbb{C}$, $TE(\lambda)=\phi(\lambda) E(\lambda)$,

\item[{\rm (b)}]\ For each $\lambda, \mu\in\mathbb{C}$, $E(\lambda )E(\mu)= E(\lambda+\mu)$,

\item[{\rm (c)}]\ For each subset $\Lambda$ of $\mathbb{C}$ supporting an accumulation point in $\mathbb{C}$, the set $\{ E(\lambda):\ \lambda\in\Lambda \}$ has dense linear span in $X$, and

\item[{\rm (d)}]\ The function $\phi$ satisfies $|\phi (0)|=1$,  $\phi''(0)\phi (0) \ne \phi'(0)^2$, and $\phi'(0)\ne 0$.

\end{enumerate}
Then the generic element $f=(f_n)_{n=1}^\infty$ of $X^\mathbb{N}$  freely generates a dense subalgebra of $X$ consisiting entirely {\rm (}but zero{\rm )} of hypercyclic vectors for $T$.
\end{theorem}

Taking $X=H(\mathbb{C})$, $\Phi=\phi$,   $T=\Phi(D)$, and $E(\lambda)(z):=e^{\lambda z}$ $(\lambda, z\in\mathbb{C})$ in Theorem~\ref{T:2bisbis} we have the following corollary.
\begin{corollary}\label{C:Phi'(0)}
Let $\Phi$ be of exponential type satisfying $|\Phi(0)|=1$, $\Phi''(0)\Phi (0) \ne \Phi'(0)^2$, and $\Phi'(0)\ne 0$.
Then the set of hypercyclic vectors for $\Phi (D)$ is densely strongly algebrable.
\end{corollary}

\begin{theorem}\label{T:7bisbis}
Let $X$ be a separable commutative $F$-algebra that supports a dense freely generated subalgebra, and let $T$ be an operator on $X$. Suppose there exist a function $E:\mathbb{C}\to X$ and an entire function $\phi:\mathbb{C}\to\mathbb{C}$  with $|\phi(0)|=1$ satisfying the following:
\begin{enumerate}
\item[{\rm (a)}] \ For each $\lambda\in\mathbb{C}$, $TE(\lambda)=\phi(\lambda) E(\lambda)$,

\item[{\rm (b)}]\ For each $\lambda, \mu\in\mathbb{C}$, $E(\lambda )E(\mu)= E(\lambda+\mu)$,

\item[{\rm (c)}]\ For each subset $\Lambda$ of $\mathbb{C}$ supporting an accumulation point in $\mathbb{C}$, the set $\{ E(\lambda):\ \lambda\in\Lambda \}$ has dense linear span in $X$, and

\item[{\rm (d)}]\ The integer
$
\mbox{min}\{ n\in\mathbb{N}:   \phi^{(n)}(0)\ne 0  \}
$
is odd and  $\phi$ supports an angle $\theta\in [0,2\pi)$ and positive scalars  $r,R$ so that its Phragm\'{e}n-Lindel\"{o}f indicator function  $h_\phi(\theta):=\limsup_{t\to\infty} \log |\phi (te^{i\theta})|^{\frac{1}{t}}$ satisfies
\[
\begin{cases}
\ |\phi (t e^{i\theta})| <1  &\mbox{ for $0<t<r$, and }\\
|\phi (R e^{i\theta})| > \mbox{max}\{ 1, e^{h_\phi(\theta) R} \}.
\end{cases}
\]
\end{enumerate}
Then the generic element $f=(f_n)_{n=1}^\infty$ of $X^\mathbb{N}$ freely generates a dense subalgebra of $X$ consisting entirely (but zero) of hypercyclic vectors for $T$.
\end{theorem}

We postpone the technical proofs of Theorem~\ref{T:2bisbis} and Theorem~\ref{T:7bisbis} to the next subsections and see first how to derive Theorem~\ref{T:IG} from them.
\begin{proof}[Proof of Theorem~\ref{T:IG}]
Case (a) follows from Theorem~\ref{T:7bisbis} and the proof of Corollary~\ref{C:7.3}. The remaining cases follow from Corollary~\ref{C:Phi'(0)}: By the Le\'on-M\"{u}ller Theorem we may assume in such cases that $b=0$. In Case (b) we have $\Phi(0)=1$, $\Phi''(0)\Phi(0)-\Phi'(0)^2=2a_2-a_1^2\ne 0$, and $\Phi'(0)=a_1+a\ne 0$. For Case (c) we have 
$
\Phi(z)=e^z \, P(z)
$
 where $P(z)=\prod_{n=1}^\infty E_p(\frac{z}{z_n})$ with $p\in \{0, 1\}$.  $\Phi(0)=1$ and $\Phi''(0)\Phi(0)\ne \Phi'(0)^2$ thanks to the assumption $\sum_{n=1}^\infty z_n^{-2}\ne 0$.  So it suffices to verify that $\Phi'(0)\ne 0$, and this holds precisely when $P'(0)+a \ne 0$. We consider two sub-cases: If $\sum_{n=1}^\infty |z_n|^{-1}$ converges, then $p=0$ and $P'(0) =-\sum_{n=1}^\infty z_n^{-1}\ne -a$ by our assumption. If $\sum_{n=1}^\infty |z_n|^{-1}$ diverges, then $p=1$ and $P'(0)=0\ne -a$ by our assumption on $a$. 
 \end{proof}

\subsection{Proof of Theorem~\ref{T:2bisbis}}

We use the following fact established with \cite[Proposition~{2.4}, Lemma~{2.10} and Lemma~{3.1}]{BP_algebrable}:

\begin{theorem} \label{T:{fromBP_algebrable}} {\rm (B.,Papathanasiou~\cite{BP_algebrable})}  
Let $X$ be a separable commutative $F$-algebra that supports a dense freely generated subalgebra, and let $T$ be an operator on $X$. Suppose that for each integer $N\ge 2$ and
each non-empty finite subset $A$ of $\N_0^N$ not containing the zero $N$-tuple there exists $\beta\in A$ satisfying:  
\begin{quote}
$(\ast)$ ``For each non-empty open subsets $U_1,\dots, U_N, V$ and $W$ of $X$ with $0\in W$ there exist
 $f\in U_1\times\dots\times U_N$ and $n\in\N$ so that $T^n(f^\beta)\in V$ and $
T^n(f^\alpha)\in W$  for each $\alpha\in A\setminus \{ \beta \}$''. 
\end{quote}
Then the generic element $f=(f_n)_{n=1}^\infty$ of $X^\mathbb{N}$  freely generates a dense subalgebra of $X$ consisting entirely {\rm (}but zero{\rm )} of hypercyclic vectors for $T$.
\end{theorem}

\begin{lemma}  {\rm (\cite[Lemma~{3.1}]{BP_algebrable}) } \label{L:BP:3.1}
Let  $A$ be a finite non-empty subset of $\mathbb{N}_0^N$, where $N\ge 2$.
Then there exist positive scalars $k_i$ 
 $(i=1,\dots, N)$   so that the functional
\[
(x_i)_{i=1 }^N
\mapsto \sum_{i=1}^N k_i x_i
\]
is injective on $A$. 
\end{lemma}

\begin{proof}[Proof of Theorem~\ref{T:2bisbis}]
Let $N\ge 2$ and let $A$ be a non-empty finite subset of $\mathbb{N}_0^N$ not containing the zero $N$-tuple.  It suffices to show there exists some $\beta$ in $A$ satisfying condition $(\ast)$ of Theorem~\ref{T:{fromBP_algebrable}}. 
By Lemma~\ref{L:BP:3.1}  there exists $k\in (0,\infty)^N$ so that the functional
\[
\mathbb{R}^N\to \mathbb{R}, \ x\mapsto \sum_{i=1}^N k_i x_i,
\]
is injective on $A$.  Permuting the $k_i's$ if necessary, without loss of generality we may assume that the subset
\[
A_1:=\{ \alpha\in A:\ \alpha_{1}=m \}
\]
of $A$ is non-empty, where $m:=\max\{ |\alpha|_\infty: \ \alpha\in A \}$. Hence there exists $\beta$ in $A$ with $\beta_1=m$ at which the functional above attains its strict minimum over $A_1$. Equivalently,  each $\alpha$ in $A\setminus \{ \beta \}$ with $\alpha_1=m$ satisfies
\begin{equation}\label{eq:a1}
\sum_{i=2}^N k_i (\beta_i-\alpha_i) <0.
\end{equation}
Now, let $U_1,\dots, U_N, V$ and $W$ be non-empty open subsets of $X$ be given, with $0\in W$. It remains to show that there exists $f$ in $U_1\times\dots\times U_N$ and a positive integer $n$ so that 
\[
\begin{aligned}
T^n(f^\beta)&\in V     \\ 
T^n(f^\alpha)&\in W \ \mbox{  for each $\alpha\in A\setminus \{ \beta \}$.}
\end{aligned}
\]
By $(d)$, taking $w_0=0$ in Lemma~\ref{L:1}(b) there exists some $w\in \mathbb{C}\setminus \{ 0 \}$  so that
\[
G:[-1,1]\to \mathbb{R}, \ G(t)=\log |\phi (tw)|,
\]
is strictly convex and strictly increasing. In particular, since $|\phi(0)|=1$ we know that $G<0$ on $[-1,0)$ and $G>0$ on $(0,1]$. Now, let $d_A:=\mbox{max}\{ \sum_{i=1}^N \alpha_i:\ \ \alpha\in A \}$, and pick $a\in (0, \frac{1}{2 d_A})$ and 
$b\in (0, \frac{a}{2 d_A})$ and define
\[
\Lambda :=[ -2aw, -aw] \ \ \ \mbox{ and } \ \ \ \Gamma := [bw, 2bw].
\]
Then we have
\begin{equation} \label{eq:4.2}
\sum_{k=1}^s \Lambda \subset [-d_A 2aw, -aw] \subset \phi^{-1}(\mathbb{D}) \ \ \ \ (1\le s\le d_A)
\end{equation}
and
\begin{equation} \label{eq:4.3}
\begin{aligned}
\sum_{k=1}^s \Lambda +\sum_{k=1}^d \Gamma &\subset [-d_A 2aw, -aw] + [bw, d_A2bw] \\
&\subset [-d_A2a+b, -a+d_A2b]w \\
&\subset (-1, -a+d_A2b]w\\
&\subset (-1, 0) w \subset
\phi^{-1}(\mathbb{D}) \ \ \ \ (1\le s<d_A, \mbox{ and } 1\le d \le d_A).
\end{aligned}
\end{equation}
Now, since each of $\Lambda$ and $\Gamma$ has accumulation points in $\mathbb{C}$, condition $(c)$ ensures that there exist a positive integer $p$  and non-zero scalars $a_{i,j}, b_j, \lambda_{i,j}, \gamma_{j}$ with $\lambda_{i,j}\in \Lambda$, $\gamma_j\in \frac{1}{m}\Gamma$
$(1\le i\le N, \ 1\le j\le p)$ so that
\[
\begin{aligned}
B&:=\sum_{j=1}^p b_j E(m\gamma_j)\in V \\
L_i&:=\sum_{j=1}^p a_{i,j} E(\lambda_{i,j})\in U_i \ \ (1\le i\le N).
\end{aligned}
\]
Moreover, we may assume that $\gamma_1,\dots, \gamma_p$ are pairwise distinct.
Next, for each positive integer $n$ let
\[
R_n:= \sum_{j=1}^p c_j E(\gamma_j),
\]
where for each $(1\le j\le p)$ the scalar $c_j=c_j(n)$ is a solution of
\begin{equation} \label{eq:4.4}
z^m= b_j \, \frac{n^{\sum_{i=2}^N k_i\beta_i} }{ \phi (m\gamma_j )^n}.
\end{equation}
Notice that $c_j=c_j(n)\to 0$ as $n\to \infty$, since $|\phi (m\gamma_j)|>1$. So letting
\[
\begin{cases}
f_1=f_{1,n}:= L_1+R_n \\
f_i=f_{i,n}:= L_i+\frac{1}{n^{k_i}} \ \ \ (2\le i\le N),
\end{cases}
\]
we have that $f=(f_1,\dots f_N)\in U_1\times\dots\times U_N$ whenever $n$ is large enough.
Now, for each $\alpha$ in $A$ let $\mathcal{I}_\alpha$ denote the set of elements $(u, v, \ell)$
in $(\N_0^p)^N\times \N_0^p\times \N_0^{N-1}$ for which $|u_1|+|v|=\alpha_1$ and $|u_i|+\ell_i=\alpha_i$ $(2\le i\le N)$.
Then for each $\alpha$ in $A$ and each positive integer $n$ we have
\[
\begin{aligned}
T^n(f^\alpha)&= T^n( \prod_{i=1}^N f_i^{\alpha_i}) \\
&=\sum_{(u,v,\ell)\in \mathcal{I}_\alpha} X(u, v, \ell, n)\, E(\lambda \cdot u + \gamma \cdot v),
\end{aligned}
\]
where 
 each $X_\alpha(u,v,\ell, n)$ is given by
\begin{equation} \label{eq:3.12}
X_\alpha(u,v,\ell,n) = {\alpha_{1}\choose{ u_1 v}  } \prod_{i=2}^N {\alpha_i\choose{ u_i\, \ell_i }}   \ a^u c^v  \ 
 \frac{ (\phi (\lambda\cdot u + \gamma\cdot v))^n}{ n^{\sum_{s=2}^N k_s \ell_s}}, 
\end{equation}
where
\[
\begin{aligned}
\lambda\cdot u+\gamma\cdot v&= \sum_{i=1}^N\sum_{j=1}^p \lambda_{i,j} u_{i,j}  + \sum_{j=1}^p \gamma_j v_j,\\
a^u&=\prod_{ 1\le i \le N, \  1\le j \le p }
 a_{i,j}^{u_{i,j}}, \mbox{ and }\\ 
c^v&= \prod_{j=1}^p c_j^{v_j}.
\end{aligned}
\]

So it suffices to show that for each $\alpha\in A\setminus\{ \beta \}$
\begin{equation}  \label{eq:19}
X_\alpha (u, v, \ell, n)\underset{n\to\infty}{\to} 0 \ \ \ \ \ ((u,v,\ell)\in \mathcal{I}_\alpha)
\end{equation}
and that for each $(u,v,\ell)\in\mathcal{I}_\beta$
\begin{equation}  \label{eq:20}
\lim_{n\to\infty} X_\beta(u,v,\ell, n)=\begin{cases}
b_j     &
\mbox{ if $(u, v)\in \{ (0, m e_j): j=1,\dots, p\}$ }\\
0 &\mbox{ otherwise,}
\end{cases}
\end{equation}
where $\{ e_1,\dots ,e_p\}$ is the standard basis of $\mathbb{C}^p$.
Now, let $\alpha\in A$ and let $(u,v,\ell)\in\mathcal{I}_\alpha$ be given. Notice that by $\eqref{eq:4.4}$ and $\eqref{eq:3.12}$ we have
\begin{equation} \label{eq:4.8}
\left| X_\alpha(u,v,\ell, n)\right|  \le (\mbox{Constant}) \ n^{\sum_{s=2}^N (\frac{|v|}{m} k_s\beta_s - k_s\ell_s)} \
\Theta_{u,v}^n,
\end{equation}
where
\[
\Theta_{u,v}=\frac{ |\phi (\lambda\cdot u+ \gamma\cdot v)| }{ \prod_{j=1}^p   | \phi(m\gamma_j)|^{\frac{v_j}{m}} }.
\]

We consider four cases.
\vspace{.1in}

Case 1: $1\le |u|< d_A$.  \ By $\eqref{eq:4.2}$ and $\eqref{eq:4.3}$ we have \[\lambda\cdot u+\gamma\cdot v\in \sum_{k=1}^{|u|} \Lambda +\sum_{k=1}^{|v|} \Gamma \subset \phi^{-1}(\mathbb{D}).\] Since $|\phi (m\gamma_j)|>1$ for each $1\le j\le p$, we have $\Theta_{u,v}\in [0, 1)$ and by $\eqref{eq:4.8}$ we have $X_\alpha (u, v, \ell, n)\underset{n\to\infty}{\to} 0$.
\vspace{.1in}

Case 2: $|u|=d_A$. So $|v|=0$ in this case, and 
\[
\lambda\cdot u+\gamma\cdot v = \gamma\cdot v \in \sum_{k=1}^{d_A} \Lambda \in \phi^{-1}(\mathbb{D})
\]
by $\eqref{eq:4.2}$, and  arguing as in Case 1 we have
$
X_\alpha (u, v, \ell, n)\underset{n\to\infty}{\to} 0 
$.
\vspace{.1in}

Case 3: $|u|=0$ and $|v|_\infty < m$.   In this case we have $u_1=\dots = u_N=0$, and hence $|v|=\alpha_1<m$ and 
$\ell_i=\alpha_i$ for each $2\le i\le N$. Now, notice that
\[
\gamma \cdot v = 0 \frac{m-|v|}{m} + m \gamma_1 \frac{v_1}{m}+\dots + m\gamma_p \frac{v_p}{m}
\]
is a non-trivial convex combination of the points $0, m\gamma_1,\dots, m\gamma_p$ in $[0, w]$. Since $|\phi(0)|=1$, we have 
\[
\Theta_{u,v}= \frac{|\phi (\gamma\cdot v)|}{|\phi(0)|^\frac{m-|v|}{m} \, \prod_{j=1}^p |\phi(m\gamma_j)|^\frac{v_j}{m}} < 1
\]
thanks to the strict convexity of $t\mapsto \log |\phi(tw)|$ on $[-1,1]$, and again by $\eqref{eq:4.8}$ we have
$
X_\alpha (u, v, \ell, n)\underset{n\to\infty}{\to} 0
$.
\vspace{.1in}

Case 4:  $|u|=0$ and $|v|_\infty=m$.   Suppose first that $v\notin\{ me_1,\dots, me_p\}$. Then 
$\lambda\cdot u+\gamma\cdot v=\gamma\cdot v$ is a nontrivial convex combination of the points $m\gamma_1,\dots, m\gamma_p$ in $[0,w]$, and again $\Theta_{u,v}\in[0,1)$ thanks to the strict convexity of the map $t\mapsto \log|\phi(tw)|$ on $[0,1]$, so $
X_\alpha (u, v, \ell, n)\underset{n\to\infty}{\to} 0 
$ by $\eqref{eq:4.8}$.

Finally, suppose $v=me_j$ for some fixed $j\in\{1,\dots, p\}$. Here we have $\Theta_{u,v}=1$, $\alpha_1=m$, and $\ell_i=\alpha_i$ for each $i\in \{2,\dots, N\}$, as $u_1=\dots=u_N=0$.  So by  $\eqref{eq:4.4}$ and  $\eqref{eq:3.12}$ we have
\[
\begin{aligned}
X_\alpha(u,v,\ell,n) &= {\alpha_{1}\choose{ u_1 v}  } \prod_{i=2}^N {\alpha_i\choose{ u_i\, \ell_i }}   \ a^u c^v  \ 
 \frac{ (\phi (\lambda\cdot u + \gamma\cdot v))^n}{ n^{\sum_{s=2}^N k_s \ell_s}} \\
 &= c_j^m \frac{\phi(m\gamma_j)^n}{n^{\sum_{s=2}^N k_s \alpha_s}} \\
 &= b_j   n^{\sum_{i=2}^N (k_i\beta_i-k_i\alpha_i)}.
 \end{aligned}
\]
So if $\alpha\ne \beta$ we have
\[
X_\alpha(u,v,\ell, n)= b_j   n^{\sum_{i=2}^N (k_i\beta_i-k_i\alpha_i)}\underset{n\to\infty}{\to} 0
\]
thanks to $\eqref{eq:a1}$, while if $\alpha=\beta$ we have \[
X_\alpha(u, v,\ell, n)=X_\beta(0, me_j, \ell, n)=b_j\underset{n\to\infty}{\to}  b_j.\]
That is, $\eqref{eq:19}$ and $\eqref{eq:20}$ are satisfied, and the proof is now complete.
\end{proof}

\subsection{Proof of Theorem~\ref{T:7bisbis}}

\begin{proof}[Proof of Theorem~\ref{T:7bisbis}]
Let $N\ge 2$ and let $A$ be a non-empty finite subset of $\mathbb{N}_0^N$ not containing the zero $N$-tuple. 
As in the proof of Theorem~\ref{T:2bisbis}, we may assume that
 the subset
\[
A_1:=\{ \alpha\in A:\ \alpha_{1}=m \}
\]
of $A$ is non-empty, where $m:=\max\{ |\alpha|_\infty: \ \alpha\in A \}$, and that there exist $k\in (0,\infty)^N$ and $\beta$ in $A_1$ so that
\begin{equation}\label{eq:a1*}
\sum_{i=2}^N k_i (\beta_i-\alpha_i) <0
\end{equation}
for each $\alpha\in A_1$.
Now, let $U_1,\dots, U_N, V$ and $W$ be non-empty open subsets of $X$ be given, with $0\in W$. By Theorem~\ref{T:{fromBP_algebrable}}, it suffices to show that there exist $f$ in $U_1\times\dots\times U_N$ and a positive integer $n$ so that 
\begin{equation} \label{eq:goal}
\begin{aligned}
T^n(f^\beta)&\in V     \\ 
T^n(f^\alpha)&\in W \ \mbox{  for each $\alpha\in A\setminus \{ \beta \}$.}
\end{aligned}
\end{equation}
By (d) and Remark~\ref{R:T:ma^{m-1}}, as in the proof of Theorem~\ref{T:ma^{m-1}}  we may get $w_0, z_0\in\mathbb{C}$ with $w_0\in \{ tz_0: t>1\}$ so that

\begin{quote} 
\begin{enumerate}
\item[{\rm (i)}]\  
$|\phi (w_0)|>1$ and  $|\phi|<1$ on $[-z_0,0)\cup (0,z_0]$, 

\item[{\rm (ii)}]\ $|\phi(dw_0)| < |\phi(w_0)|^d$ for each $d\in\{ 2,\dots, m\}$, and

\item[{\rm (iii)}]\ The map $t\mapsto |\phi (w_0+tz_0)|$ is strictly increasing on $(-1, 1)$.
\end{enumerate}
\end{quote}

Let $\gamma_1\in (0, \frac{z_0}{m})\cap D(0,1)$ close enough to zero so that 
\begin{quote}
\begin{enumerate}
\item[{\rm (i')}]\ \ \ $1< | \phi(w_0+(m-1)\gamma_1)|$,
\item[{\rm (ii')}]\ for each $d\in \{2,\dots, m\}$ and $s\in \{ 0,\dots, m-2\}$
\[
|\phi (dw_0+s\gamma_1)|^\frac{1}{d} < |\phi (w_0+(m-1)\gamma_1)|, \ \ \ \ \mbox{ and }
\]
\item[{\rm (iii')}]\ the function
 $t\mapsto |\phi(w_0+t\gamma_1)|$ is strictly increasing on the interval $(-d_A-1, d_A+1)$, where $d_A=\mbox{max}_{\alpha\in A}|\alpha|$.
\end{enumerate}
\end{quote}
By $(i')$ and $(ii')$, there exists $\delta\in (0,1)$ small enough so that
\begin{quote}
\begin{enumerate}
\item[(1)]\ $1<|\phi (\lambda +(m-1)\gamma_1)|$ for each $\lambda\in D(w_0, \delta)$, and

\item[(2)]\ for each $\lambda, \lambda_1',\dots, \lambda_d'$ in $D(w_0,\delta)$ and $z$ in $D(0, \delta)$ and each $(d,s)\in \{ 2,\dots, m\}\times\{ 0,\dots, m-2\}$,
\[
|\phi (\lambda_1'+\cdots+\lambda_d'+s\gamma_1+z)|^\frac{1}{d}< |\phi(\lambda +(m-1)\gamma_1)|.
\]
\end{enumerate}
\end{quote}
Reducing $\delta>0$ if necessary, by $(iii')$ we may further assume that
\begin{quote}
\begin{enumerate}
\item[(3)]\ for each $ \lambda$ in $(w_0-\delta\gamma_1, w_0+\delta \gamma_1)$, the function
\[
t\mapsto |\phi (\lambda+t\gamma_1)|  
\]
is strictly increasing on $[-d_A,d_A]$.
\end{enumerate}
\end{quote}
Let $a_{1,1}:=1$, $\gamma_{1,1}:=\gamma_1$, and 
\begin{equation} \label{eq:L*}
 R > \frac{d_A}{\delta}.
\end{equation}
By (c) there exist  $b_1,\dots, b_p, a_{1,2},\dots, a_{1,p}\in \mathbb{C}\setminus\{ 0\}$,  pairwise distinct $\lambda_1,\dots, \lambda_p\in (w_0-\delta\gamma_1, w_0+\delta\gamma_1)$, and 
\begin{equation} \label{eq:L**}
\gamma_{1,2},\dots,\gamma_{1,p}\in (0,\frac{\gamma_1}{R})
\end{equation} so that
\begin{equation} 
\begin{aligned}\label{eq:0.5*}
B&:=\sum_{j=1}^p b_j\ E(\lambda_j+(m-1)\gamma_1)\in V \\
L_1&:=\sum_{j=1}^p a_{1,j}\ E(\gamma_{1,j})\in U_1.
\end{aligned}
\end{equation}
Now, let $\eta:=\mbox{min}_{1\le j\le p}|\gamma_{1,j}|$.  Again by $(c)$, there exist an integer $q\ge p$ and  non-zero scalars $a_{i,j}$, $ \gamma_{i,j}$ 
 $((i,j)\in \{ 2,\dots, N\}\times \{ 1,\dots, q\} )$ with
\begin{equation} \label{eq:L***}
 \gamma_{i,j}\in [-z_0, 0)\cap D(0, \frac{\eta}{R}) 
\end{equation}
so that
\begin{equation} \label{eq:1*}
L_i:=\sum_{j=1}^p a_{i,j}\ E(\gamma_{i,j})\in U_i \ \ (2\le i\le N).
\end{equation}
Enlarging $p$ if necessary (and choosing some new $\gamma_{1,j}$'s in $(\eta \frac{\gamma_1}{|\gamma_1|}, \frac{\gamma_1}{R})$ and corresponding small $a_{1,j}$'s in $\mathbb{C}\setminus\{ 0\}$ so that $\eqref{eq:0.5*}$ is preserved), we may assume that $q=p$. 
Next, for each positive integer $n$ let
\[
R_n:= \sum_{j=1}^p c_j\ E(\lambda_j),
\]
where for each $j\in\{1,\dots, p\}$ the scalar $c_j=c_j(n)$ is a solution of
\begin{equation} \label{eq:4.4*}
b_j= m\, c_j \  \frac{n^{\sum_{i=2}^N k_i\beta_i} }{ \phi (\lambda_j+(m-1)\gamma_1 )^n}.
\end{equation}
Notice that $c_j=c_j(n)\to 0$ as $n\to \infty$, as $|\phi (\lambda_j+(m-1)\gamma_1)|>1$ by $(1)$. So letting
\[
\begin{cases}
f_1=f_{1,n}:= L_1+R_n \\
f_i=f_{i,n}:= L_i+\frac{1}{n^{k_i}} \ \ \ (2\le i\le N),
\end{cases}
\]
we have that $f=(f_1,\dots f_N)\in U_1\times\dots\times U_N$ whenever $n$ is large enough.
Again, for each $\alpha$ in $A$ let $\mathcal{I}_\alpha$ denote the set of elements $(u, v, \ell)$
in $(\N_0^p)^N\times \N_0^p\times \N_0^{N-1}$ for which $|u_1|+|v|=\alpha_1$ and $|u_i|+\ell_i=\alpha_i$ $(2\le i\le N)$.
Then for each $\alpha$ in $A$ and each positive integer $n$ we have
\begin{equation}\label{eq:2*}
\begin{aligned}
T^n(f^\alpha)&= T^n( \prod_{i=1}^N f_i^{\alpha_i}) \\
&=\sum_{(u,v,\ell)\in \mathcal{I}_\alpha} X_\alpha (u, v, \ell, n)\, E(\lambda \cdot u + \gamma \cdot v),
\end{aligned}
\end{equation}
where 
\begin{equation} \label{eq:3.12*}
\begin{aligned}
X_\alpha(u,v,\ell,n) &= {\alpha_{1}\choose{ u_1 v}  } \prod_{i=2}^N {\alpha_i\choose{ u_i\, \ell_i }}   \ a^u c^v  \ 
 \frac{ (\phi (\gamma\cdot u + \lambda \cdot v))^n}{ n^{\sum_{s=2}^N k_s \ell_s}}, \\
\lambda\cdot u+\gamma\cdot v&= \sum_{i=1}^N\sum_{j=1}^p \lambda_{i,j} u_{i,j}  + \sum_{j=1}^p \gamma_j v_j,\\
a^u&=\prod_{ 1\le i \le N, \  1\le j \le p }
 a_{i,j}^{u_{i,j}}, \mbox{ and }\\ 
c^v&= \prod_{j=1}^p c_j^{v_j}.
\end{aligned}
\end{equation}

Notice that for each $k\in \{ 1,\dots, p\}$ the element $(u,v,\ell)\in \mathcal{I}_\beta$ given by $v=e_k$,
$u_1=(m-1)e_1$, $|u_2|=\dots =|u_N|=0$, and $\ell=(\beta_i)_{i=2}^N$ satisfies
\begin{equation} \label{eq:3*}
\begin{cases}
X_\beta(u,v,\ell)=b_k \ \mbox{ and } \\
E(\gamma\cdot u +\lambda\cdot v)=E(\lambda_k+(m-1)\gamma_1).
\end{cases}
\end{equation}
Hence by $\eqref{eq:1*}$ and $\eqref{eq:2*}$  it suffices to show that 
\begin{equation}  \label{eq:19*}
X_\alpha (u, v, \ell, n)\underset{n\to\infty}{\to} 0 \ \ \ \ \ 
\end{equation}
except for the $p$ elements $(u, v,\ell)$ of $\cup_{\alpha\in A} \mathcal{I}_\alpha$ considered in $\eqref{eq:3*}$.
To this end, let $\alpha\in A$ and let $(u,v,\ell)\in\mathcal{I}_\alpha$ be fixed, not of the form considered in $\eqref{eq:3*}$. Notice that 
\begin{equation} \label{eq:4.8*}
\left| X_\alpha(u,v,\ell, n)\right|  \le  K \ n^{\sum_{s=2}^N (\frac{|v|}{m} k_s\beta_s - k_s\ell_s)} \
\Theta_{u,v}^n,
\end{equation}
where
\[
\Theta_{u,v}=\frac{ |\phi (\gamma\cdot u+ \lambda \cdot v)| }{ \prod_{j=1}^p   | \phi(\lambda_j+(m-1)\gamma_1)|^{\frac{v_j}{m}} }
\]
and where $K$ is a constant that does not depend on $n$.
In particular, $\eqref{eq:19*}$ holds whenever $\Theta_{u,v}<1$.
We consider the following cases:

\vspace{.1in}

Case 1:  $|v|=0$.   So $|u_1|=\alpha_1-|v|=\alpha_1$ in this case.
If $|u|=0$, then  $\alpha_1=0$ and  $\ell=(\alpha_i)_{i=2}^N$, and since $A$ does not contain the zero N-tuple we must have $\ell_i=\alpha_i>0$ for some $i\in \{ 2,\dots, N\}$. Since $|\Phi(u\cdot \gamma)|=|\phi(0)|\le 1$,
by $\eqref{eq:4.8*}$  and $\eqref{eq:a1*}$ we have
\[
|X_\alpha(u,v,\ell,n)| = K \ n^{-\sum_{s=2}^N\ell_s k_s} \ |\phi (u\cdot \gamma)|^n \underset{n\to\infty}{\to} 0.
\]
So we assume $|u|\ne 0$.  If $|u_1|\ne 0$, then by $\eqref{eq:L*}$, $\eqref{eq:L**}$ and $(i)$
\[
\begin{aligned}
u\cdot \gamma &=u_{1,1}\gamma_1+\sum_{j=2}^p u_{1,j}\gamma_{1,j} + \sum_{i=2}^N\sum_{j=1}^p u_{i,j}\gamma_{i,j}\\
&\in (0, \frac{|u|-u_{1,1}}{R}+u_{1,1})\gamma_1 \subset (0, z_0]\subset \phi^{-1}(\mathbb{D}),
\end{aligned}
\]
so $\Theta_{u,v}=|\phi( u\cdot \gamma)|<1$. On the other hand, if $|u_1|=0$, we have $|u_2|+\dots+|u_N|=|u|\ne 0$ and thus by $\eqref{eq:L*}$, $\eqref{eq:L***}$ and $(i)$ we have $u\cdot \gamma\in [-z_0, 0) \subset\phi^{-1}(\mathbb{D})$ and hence
$\Theta_{u,v}=|\phi( u\cdot \gamma)|<1$.
\vspace{.1in}

Case 2: $|v|\in \{2,\dots, m\}$.
Here we have $u_{1,1}\le |u_1|=\alpha_1-|v|\le m-2$. Also, $z:=u\cdot\gamma - u_{1,1}\gamma_1$ has size
\[
|z|=|\sum_{j=2}^pu_{1,j}\gamma_{1,j}+\sum_{i=2}^N\sum_{j=1}^p u_{i,j}\gamma_{i,j}| \le \frac{|\alpha|-|v|}{R}|\gamma_1|\le \frac{d_A-2}{R}|\gamma_1|<\delta.
\]
by $\eqref{eq:L*}$.
So by $(2)$ we have
\[
\begin{aligned}
|\phi (u\cdot \gamma+\lambda\cdot v)|&=|\phi( u_{1,1}\gamma_1+z+v\cdot\lambda)|\\
& < |\phi (\lambda_j+(m-1)\gamma_1)|^p \ (1\le j\le p),
\end{aligned}
\]

and thus

\[
\Theta_{u,v}=\frac{|\phi(u\cdot\gamma+\lambda\cdot v)|}{\prod_{j=1}^p |\phi (\lambda_j+(m-1)\gamma_1)|^{v_j} }= \prod_{j=1}^p  \frac{ |\phi(u\cdot\gamma+\lambda\cdot v)|^\frac{v_j}{d} }{ |\phi (\lambda_j+(m-1)\gamma_1)|^{v_j}} <1 . 
\]
\vspace{.05in}

Case 3:  $|v|=1$. So $v=e_k$ for some $k\in \{ 1,\dots, p\}$, and $0\le u_{1,1} \le |u_1|=\alpha_1-1$.

If $|u_1|=0$, then by $\eqref{eq:L***}$ we have $u\cdot\gamma=-t\gamma_1$ for some $t\in [0, m-1)$. Hence
 by (3)
\[
|\phi(u\cdot\gamma+v\cdot\lambda)|<|\phi(\lambda_k+(m-1)\gamma_1)|,
\]
giving $\Theta_{u,v}<1$.  So assume that $|u_1|>0$. Now, if  $u_{1,1}< m-1$, then by $\eqref{eq:L*}$ and $\eqref{eq:L**}$
\[u\cdot\gamma+v\cdot\lambda =\lambda_k+t\gamma_1\]
for some $t\in (0, m-1)$, giving again by (3)
\[
|\phi(u\cdot\gamma+v\cdot\lambda)|<|\phi(\lambda_k+(m-1)\gamma_1)|,
\]
and thus $\Theta_{u,v}<1$. Finally, if $u_{1,1}=m-1$, then we must have $\alpha_1=m$ and $u_1=(m-1)e_1$. Here we have two possibilities: either $|u_i|>0$ for some $2\le i\le N$ or else $|u_2|=\dots=|u_N|=0$.  In the former case we have by $\eqref{eq:L*}$ and $\eqref{eq:L**}$
\[
u\cdot\gamma+v\cdot\lambda =\lambda_k+t\gamma_1\]
 for some $t\in (0, m-1)$, giving $\Theta_{u,v}<1$.
In the latter case we have $\ell=(\alpha_i)_{i=2}^N$ and $u\cdot\gamma+v\cdot\lambda =\lambda_k+(m-1)\gamma_1$, giving $\Theta_{u,v}=1$. Since $(u,v,\ell)$ is not of the form considered in   $\eqref{eq:3*}$, then $\alpha\ne \beta$. Hence by $\eqref{eq:4.8*}$  and $\eqref{eq:a1*}$ we have
\[
|X_\alpha(u,v,\ell, n)| \le K \ n^{\sum_{s=2}^N (k_s\beta_s-k_s\alpha_s)}\underset{n\to\infty}{\to} 0,
\]
and $\eqref{eq:19*}$ holds.
\end{proof}


\end{document}